\documentclass[12pt]{amsart}
\usepackage{color}
\usepackage{amsmath}
\usepackage{amsthm}
\usepackage{amssymb}
\usepackage{graphicx}
\usepackage{amsfonts}
\usepackage{mathrsfs}
\usepackage{parskip}
\usepackage{mathdots}
\usepackage{color}

\numberwithin{equation}{section}
\theoremstyle{plain}
\newtheorem{Proposition}[equation]{Proposition}

\newtheorem{Theorem}[equation]{Theorem}
\newtheorem{Lemma}[equation]{Lemma}
\theoremstyle{definition}

\newtheorem{Example}[equation]{Example}
\newtheorem{Remark}[equation]{Remark}
\newtheorem{Question}[equation]{Question}

\def\HH{\mathscr{H}}
\def\MM{\mathscr{M}}
\def\C{\mathbb{C}}

\def\D{\mathbb{D}}
\def\T{\mathbb{T}}

\title{Concrete examples of $\HH(b)$ spaces}

\author[Fricain]{Emmanuel Fricain}
 \address{Laboratoire Paul Painlev\'e, Universit\'e Lille 1, 59 655 Villeneuve d'Ascq C\'edex }
 \email{emmanuel.fricain@math.univ-lille1.fr}

\author[Hartmann]{Andreas Hartmann}
\address{Institut de Math\'ematiques de Bordeaux, Universit\'e Bordeaux 1, 351 cours de la Lib\'eration 33405 Talence C\'edex, France}
\email{Andreas.Hartmann@math.u-bordeaux1.fr}

\author[Ross]{William T. Ross}
	\address{Department of Mathematics and Computer Science, University of Richmond, Richmond, VA 23173, USA}
	\email{wross@richmond.edu}

\thanks{This work was initiated while the first two authors were staying at the University of Richmond. These authors would like to thank that institution for the great hospitality. Work supported by Labex CEMPI (ANR-11-LABX-0007-01).}



\begin{document}

\begin{abstract}
In this paper we give an explicit description of de Branges-Rovnyak spaces $\HH(b)$ when $b$ is of the form $q^{r}$, where $q$ is a rational outer function in the closed unit ball of $H^{\infty}$ and $r$ is a positive number.\end{abstract}

\maketitle

\section{Introduction}

The purpose of this paper is to explicitly describe the elements of the de Branges-Rovnyak space $\HH(b)$ for certain $b \in \mathbf{b}(H^{\infty})$. Here $H^{\infty}$ denotes the space of bounded analytic functions on the open unit disk $\D$ normed by $\|f\|_{\infty} := \sup_{z \in \D} |f(z)|$,
and
$\mathbf{b}(H^{\infty}) := \{g \in H^{\infty}: \|g\|_{\infty} \leqslant 1\}$ is the closed unit ball in $H^{\infty}$
and, for $b \in \mathbf{b}(H^{\infty})$, the {\em de Branges-Rovnyak space} $\HH(b)$ is the reproducing kernel Hilbert space of analytic functions on $\D$ whose kernel is 
$$k^{b}_{\lambda}(z) := \frac{1 - \overline{b(\lambda)} b(z)}{1 - \overline{\lambda} z}, \qquad \lambda, z \in \D.$$

Besides possessing a fascinating internal structure \cite{Sa}, $\HH(b)$ spaces play an important role in several aspects of function theory and operator theory, most importantly, in the model theory for many types of contraction operators \cite{MR0244795,MR0215065}. 

Despite the important role $\HH(b)$ spaces play in operator theory, the exact contents of $\HH(b)$ often remain mysterious. What functions belong to $\HH(b)$? Certainly the kernel functions $k^{b}_{\lambda}, \lambda \in \D$, do (and have dense linear span). What else?

In this paper, we give a precise description of the elements of $\HH(b)$ for certain relatively simple $b$, namely positive powers of rational outer functions.  Our description needs the following set up. If $b \in \mathbf{b}(H^{\infty})$ is a non-extreme point of $\mathbf{b}(H^{\infty})$, equivalently, $\log (1 - |b|) \in L^1(\T, m)$
(where $\T := \{\zeta \in \C: |\zeta| = 1\}$ and $m$ Lebesgue measure on $\T$ normalized so that $m(\T) = 1$), then there exists a unique outer function $a \in \mathbf{b}(H^{\infty})$, called the {\em Pythagorean mate} for $b$, such that $a(0) > 0$ and $|a|^2 + |b|^2 = 1$ almost everywhere on $\T$. The pair $(a, b)$ is said to be a {\em Pythagorean pair}. 

Our first observation says that in certain situations $\HH(b^r)$ does not depend on $r>0$.

\begin{Theorem}\label{Thm1}
\hfill
\begin{enumerate}
\item Suppose $b\in \mathbf{b}(H^{\infty})$ is outer. The following are equivalent:
\begin{enumerate}
\item For any $r > 0$ we have $\HH(b^r) = \HH(b)$ as sets. 
\item $\HH(b^2) = \HH(b)$ as sets. 
\item $b \HH(b) \subset \HH((b)$. 
\end{enumerate}
\item If $b$ is non-extreme, i.e., $\log(1 - |b|) \in L^1(\T)$, with Pythagorean mate $a$, then conditions (a), (b), and (c) are equivalent to the condition 
\begin{equation}\label{Thm-CP}
\inf\{|a(z)| + |b(z)|: z \in \D\} > 0.
\end{equation}
\item If $b$ extreme, i.e., $\log (1 - |b|) \not \in L^1(\T)$, then conditions (a), (b), and (c) are equivalent to the condition 
\begin{equation}\label{Thm-b-inv}
\mbox{$b$ is invertible in $H^{\infty}$}. 
\end{equation}
\end{enumerate}
\end{Theorem}

\begin{Remark}
\begin{enumerate}
\item Since $b$ is outer, it has no zeros on $\D$ and so we can define $b^r$ by taking any logarithm of $b$. Note that $b^r \in \mathbf{b}(H^{\infty})$.
\item Statement (a) of Theorem \ref{Thm1} says that $\HH(b^r) = \HH(b)$ as sets. Though the norms on $\HH(b^r)$ and $\HH(b)$ are different, one sees from the closed graph theorem that they are equivalent. 
\item Statement (c) of the theorem says that $b$ is a {\em multiplier} of $\HH(b)$. We refer the reader to Sarason's book \cite{Sa} for further information and references about multipliers of $\HH(b)$. 
\item By Carleson's corona theorem \cite{Garnett}, the condition \eqref{Thm-CP} is equivalent to existence of $\phi, \psi \in H^{\infty}$ so that $a \phi + b \psi = 1$ on $\D$. Such a pair $(a, b)$ satisfying this condition is called a {\em corona pair}. 
\end{enumerate}
\end{Remark}



When $b$ is a {\em rational} outer function, or any positive power of a rational function (which is necessarily non-extreme (see Lemma \ref{mate})), we obtain the following complete description of 
$\HH(b)$ involving the derivatives of the reproducing kernels. Indeed, when $b=q^r$, where $q$ is outer and rational and $r>0$, we set
$$v_{r,\lambda}^\ell(z) := \frac{d^\ell}{d \overline{\lambda}^\ell}k_{\lambda}^{q^r}(z)=\frac{d^\ell}{d\overline{\lambda}^\ell}\left(\frac{1-\overline{q^r(\lambda)}q^r(z)}{1-\overline{\lambda}z}\right),$$
for any $z\in\D$, $\lambda\in\D^-$, and $\ell \geqslant 0$.
We let $H^2$ denote the classical Hardy space \cite{Garnett}. By means of the F\'{e}jer-Riesz theorem (see Section \ref{Ex}), one can prove that if $q$ is a rational function then so is its Pythagorean mate $a$. In this case, also notice  that for $\zeta \in \T$ we have $|q(\zeta)| = 1$ if and only if $a(\zeta) = 0$.


\begin{Theorem}\label{MainThm}
Suppose $q \in \mathbf{b}(H^{\infty})$ is a rational outer function and $r$ is a positive real number. Then 
\begin{enumerate}
\item $\HH(q^{r}) = \HH(q)$ as sets.
\item If $a$ is the Pythagorean mate for $q$ and $a$ has distinct zeros $\zeta_1, \ldots, \zeta_n$ on $\T$ with corresponding multiplicities $m_1, \ldots, m_n$, then 
\begin{enumerate}
\item the functions $v_{r,j}^\ell:=v_{r,\zeta_j}^\ell$ are well-defined and
belong to $\HH(q^r)$ for $1\leqslant j\leqslant n$ and $0\leqslant \ell \leqslant m_j-1$. Moreover, they are orthogonal to 
$$aH^2=\left(\prod_{j = 1}^{n} (z - \zeta_j)^{m_j} \right) H^2.$$
\item $\HH(q^r)$ is equal to 
$$\left(\prod_{j = 1}^{n} (z - \zeta_j)^{m_j} \right) H^2 \oplus \bigvee 
\left\{ 
v_{r,j}^\ell: 0 \leqslant \ell  \leqslant m_j - 1, 1 \leqslant j \leqslant n\right\},$$
where the orthogonal decomposition is in terms of the inner product in $\HH(q^r)$. 
\end{enumerate}
\end{enumerate}
\end{Theorem} 

Writing $v_j^\ell=v_{1,j}^\ell$, the theorem above implies that $\HH(q^r)$ is equal to 
$$\left(\prod_{j = 1}^{n} (z - \zeta_j)^{m_j} \right) H^2 \dotplus \bigvee \left\{
v_{j}^\ell(z): 0 \leqslant \ell  \leqslant m_j - 1, 1 \leqslant j \leqslant n\right\},$$
where the sum is no longer necessarily orthogonal.

It was shown in  \cite{ransford}, and rediscovered in \cite{RCDR}, that 
\begin{eqnarray}\label{sardecomp}
\HH(q) = \left(\prod_{j = 1}^{n} (z - \zeta_j)^{m_j} \right) H^2 \dotplus \mathscr{P}_{N - 1},
\end{eqnarray}
where 
$N = \sum_{j = 1}^{n} m_j$,
$\mathscr{P}_{N - 1}$ is the $N$-dimensional vector space of polynomials of degree at most $N - 1$, and the sum is an algebraic direct sum (not necessarily an orthogonal one). The novelty of our result is that we can precisely identify the orthogonal complement of $ aH^2=\left(\prod_{j = 1}^{n} (z - \zeta_j)^{m_j} \right) H^2$ in $\HH(q)$
without using \eqref{sardecomp}. 

In a recent preprint, Lanucha and Nowak \cite{LanNow} examined when an $\HH(b)$ space is isomorphic to a Dirichlet type
space. Their discussion naturally leads to the situation when
$a$ is a polynomial with simple zeros on $\T$ and a similar
description of $\HH(b)$ for such $a$.


A key ingredient used to show statement (1) of Theorem \ref{MainThm}, and
an added bonus to our result, is that if $a_r$ is the Pythagorean mate for $q^r$ then 
the co-analytic Toeplitz operators $T_{\overline{a}}$ and $T_{\overline{a_r}}$ on $H^2$ have the same range, namely $\HH(q)$. 

\section{Preliminaries}

There are several equivalent definitions of the de Branges-Rovnyak space $\HH(b)$.
We can, for instance,
define it 
in the standard way \cite{Paulsen} as the reproducing kernel Hilbert space
associated with the 
(positive definite) reproducing kernel 
$$k^{b}_{\lambda}(z) := \frac{1 - \overline{b(\lambda)} b(z)}{1 - \overline{\lambda} z}, \qquad \lambda, z \in \D.$$
By definition, 
$f(\lambda) = \langle f, k^{b}_{\lambda}\rangle_b$ for all $f \in \HH(b)$ and $\lambda \in \D$, 
where $\langle\cdot,\cdot\rangle_b$ represents the scalar product in $\HH(b)$. 

The space $\HH(b)$ can also be defined as the range space $(I-T_bT_{\overline{b}})^{1/2}H^2$ equipped
with the norm which makes $(I-T_bT_{\overline{b}})^{1/2}$ a partial isometry. Here $T_{\varphi}$ is the Toeplitz
operator on $H^2$ with symbol $\varphi\in L^{\infty}(\T)$ defined by 
$$T_{\varphi}f=P_+(\varphi f), \qquad f\in H^2,$$ where $P_+$ is  the orthogonal
projection of $L^2(\T)$ onto $H^2$. The book \cite{Sa} is the classic reference for $\HH(b)$ spaces. 

When $\|b\|_{\infty} < 1$, $\HH(b)$ turns out to be a renormed version of 
$H^2$ while if $b$ is an inner function, then $\HH(b)$ turns out to be one of the classical and well-studied model spaces $H^2 \ominus b H^2$. 

When $b$ is non-extreme and $a$ is its Pythagorean mate, two important  (not necessarily closed) vector spaces of functions in $\HH(b)$  are 
$$\MM(a) := T_{a} H^2 \quad \mbox{and} \quad \MM(\overline{a}) := T_{\overline{a}} H^2.$$
It follows from the Douglas factorization theorem and the operator inequalities 
\begin{equation}\label{DouglasFact}
T_{a} T_{\overline{a}} \leqslant T_{\overline{a}} T_{a} \quad \mbox{and} \quad T_{\overline{a}} T_{a} = I - T_{\overline{b}} T_{b} \leqslant I - T_{b} T_{\overline{b}}
\end{equation}
 that 
$\MM(a) \subset \MM(\overline{a}) \subset \HH(b)$ (see \cite[p.~24]{Sa}).

For technical reasons, we will make use of the space $\HH(\overline{b})$ which, for any $b \in \mathbf{b}(H^{\infty})$, is defined similarly as with $\HH(b)$ but as the range space $(I - T_{\overline{b}} T_{b})^{1/2} H^2$. The operator inequalities from \eqref{DouglasFact} show that $\HH(\overline{b})$ is contractively contained in $\HH(b)$. 

\section{Corona pairs}

This following lemma is well-known but we record it here along with a proof for the sake of completeness and for the discussion of the examples in Section \ref{Ex}. 

\begin{Lemma}\label{mate}
Suppose $q \in \mathbf{b}(H^{\infty})$ is rational and not inner. Then $q$ is non-extreme and, if $a$ is the Pythagorean mate for $q$, then $a$ is also rational. 
\end{Lemma}

\begin{proof}
Since $q$ is rational then $q = p_1/p_2$ where $p_1$ and $p_2$ are analytic polynomials 
and  $p_2$ has no zeros on $\D^{-}$. We can, of course, choose $p_2$ such that $p_2(0)>0$. 
Since $q \in \mathbf{b}(H^{\infty})$, we see that $1 - |q(e^{i \theta})|^2 \geqslant 0$ for all $\theta$ and so 
$|p_2(e^{i \theta})|^2 - |p_1(e^{i \theta})|^2$ is a non-negative trigonometric polynomial. Furthermore, $|p_2(e^{i \theta})|^2 - |p_1(e^{i \theta})|^2$ is not the zero function since we are assuming that $q$ is not an inner function. 
By the F\'{e}jer-Riesz theorem, $|p_2(e^{i \theta})|^2 - |p_1(e^{i \theta})|^2 = |p(e^{i \theta})|^2$, where $p$ is an analytic polynomial which is zero free in $\D$ and $p(0)>0$. 

Let 
$a = p/p_2$. 
Note that $a$ is rational and zero free in $\D$, hence outer. Moreover, $a(0)>0$. 

Furthermore, on $\T$ we have
$$|a|^2  = \left|\frac{p}{p_2}\right|^2
 = \frac{|p_2|^2 - |p_1|^2}{|p_2|^2}
 = 1 - \left|\frac{p_1}{p_2}\right|^2
 = 1 - |q|^2.
$$ 
This means that $(a,q)$ is a Pythagorean pair which, in particular, implies that
$q$ is non-extreme. 
\end{proof}


\begin{Lemma}\label{simCoP}
Suppose $b \in \mathbf{b}(H^{\infty})$ is outer and $r$ is a positive real number. 
Then $b$ and $b^r$ are simultaneously 
non-extreme. Moreover, if $a_r$ is the Pythagorean mate for $b^r$, the pairs 
$(a,b)$ and $(a_r,b^r)$ are simultaneously corona.
\end{Lemma}

\begin{proof}
Since
\begin{equation}\label{E3}
\frac{1 - x^r}{1 - x} \asymp 1, \quad x\in [0,1),
\end{equation} 
we see that 
$1 - |b|^r \asymp 1 - |b|$ when $b\in \mathbf{b}(H^{\infty})$,
from which we deduce
the first part of the Lemma.

Now observe that 
\[
 \frac{|a|^2}{|a_r|^2}=\frac{1-|b|^2}{1-|b^2|^r}\asymp 1,
\]
and since $a$ and $a_r$ are outer, Smirnov's theorem (which says that if the boundary function for the quotient of two outer functions is bounded on $\T$, then $f \in H^{\infty}$), shows that $a/a_r$ is invertible in $H^{\infty}$.  
Thus both expressions 
$$\inf_{z\in\D}(|a(z)|+|b(z)|) \quad \mbox{and} \quad \inf_{z\in\D}(|a_r(z)|+|b^r(z)|)$$
 are strictly positive (or not) simultaneously. Indeed, if there is a sequence $\{z_n\}_{n \geqslant 1}$ in $\D$ such that one expression goes to $0$ then, since both $a(z_n)$ and $b(z_n)$  go to zero, 
the other expression will go to zero as well.
\end{proof}

A special situation where $b$ forms a corona pair with its Pythagorean mate is
when $b$ is rational.

\begin{Lemma}\label{a1bCP}
Suppose $q \in \mathbf{b}(H^{\infty})$ is rational and not inner. 
If $a$ is the Pythagorean mate for $q$, then 
$(a,q)$ is a corona pair.
\end{Lemma}

\begin{proof}
According to the proof of Lemma~\ref{mate}, we know that $a$ is rational, $a=p/p_2$, where $p$ and $p_2$ are polynomials, $p_2$ has no zeros in $\D^-$ and $p$ is zero free in $\D$. In particular, $a$ is analytic in an open neighborhood of $\D^{-}$ and thus has a finite number of zeros on $\T$, say $\{\zeta_1, \ldots, \zeta_n\}$. Note that, due to the identity $|a|^2 + |q|^2 = 1$ on $\T$,  the zeros of $a$ (on $\T$) must lie where $q$  is unimodular on $\T$. 

%
%
%

 Let $D_j$ be disjoint open disks with center at the zeros $\zeta_j$ of $a$ and let 
$$F = \D^{-} \setminus \bigcup_{j = 1}^{n} D_j.$$
By making the disks smaller, one can, by using the continuity of $|q|$ 
on $\D^{-}$, arrange things so that $|q| \geqslant \tfrac{1}{2}$ on each $D_j\cap\D^-$. 

Notice that $F$ is closed and omits all of the zeros of $a$ in $\D^-$ and so 
$$\inf_{z \in F} |a(z)| = \delta > 0.$$
Thus 
$$\inf_{z \in \D} (|a(z)| + |q(z)|) \geqslant \min(\tfrac{1}{2}, \delta) > 0$$
concluding the proof. 
\end{proof}


The first statement of Theorem \ref{Thm1} depends on the following two results. The first is from Sarason's book \cite[p.~62]{Sa}.

\begin{Proposition}\label{PropSa}
For $b \in \mathbf{b}(H^{\infty})$ and non-extreme, the following are equivalent:
\begin{enumerate}
\item $(a, b)$ is a corona pair; 
\item $\HH(b) = \MM(\overline{a})$.
\end{enumerate}
\end{Proposition}

The second is the following. 

\begin{Proposition}\label{PropRngToe}
If $a, a_1 \in H^{\infty}$ are two outer functions such that 
$a/a_1$ and $a_1/a$ belong to $L^{\infty}$,
then $\MM(\overline{a}) = \MM(\overline{a_1})$.
\end{Proposition}

\begin{proof}
Again, by Smirnov's theorem, we know that $a/a_1$ and $a_1/a$ belong to $H^{\infty}$. Thus 
 $T_{a/a_1}$, and hence $T_{\overline{a/a_1}}$, are invertible operators on $H^2$. From here we get 
$$\MM(\overline{a})  = T_{\overline{a}} H^2
 = T_{\overline{a_1}} T_{\overline{a/a_1}} H^2
 =T_{\overline{a_1}} H^2
 = \MM(\overline{a_1}). \qedhere
$$ 
\end{proof}

\section{$\HH(b^r) = \HH(b)$ as sets}

We are now ready to prove Theorem \ref{Thm1}.

\begin{proof}[Proof of Theorem \ref{Thm1}]
The implication $(a) \implies (b)$ is trivial. 

To show $(b) \implies (c)$ note from \cite[I-10]{Sa} we have 
$$\HH(b^2) = \HH(b) + b \HH(b).$$
But since we are assuming that $\HH(b^2) = \HH(b)$ is follows that $b \HH(b) \subset \HH(b)$. 

For the implication $(c) \implies \eqref{Thm-CP}$, we use the fact that $b$ is non-extreme and \cite[VIII-1, VIII-7]{Sa} to see that $b$ being a multiplier of $\HH(b)$ is equivalent to $(a, b)$ being a corona pair.

To show that $\eqref{Thm-CP} \implies (a)$, we proceed as follows. By Lemma 
\ref{simCoP}
we know that 
since $(a, b)$ is a corona pair, then so is $(a_r, b^r)$. 
Thus from Proposition \ref{PropSa} we see that $\HH(b) = \MM(\overline{a})$ and $\HH(b^r) = \MM(\overline{a_r})$. 
As in the proof of Lemma \ref{simCoP} $a/a_r$ and $a_r/a$ belong to $H^{\infty}$ so, by
Proposition \ref{PropRngToe}, 
we get $\MM(\overline{a}) = \MM(\overline{a_r})$. Putting this all together we get the desired set equality $\HH(b^r) = \HH(b)$. 

For the implication $(c) \implies \eqref{Thm-b-inv}$, we use the fact that $b$ is extreme and \cite[VIII-1, VIII-5]{Sa} to see that $b$ being a multiplier of $\HH(b)$ is equivalent to $b$ being an invertible element of $H^{\infty}$. 

It remains to show $\eqref{Thm-b-inv} \implies (a)$. Assuming that $b$ is invertible in $H^\infty$, we use, once again, \cite[VIII-1]{Sa} to see that $\HH(b)=\HH(\bar b)$. But, since $b$ is invertible in $H^\infty$, then so is $b^r$ and we thus also have $\HH(b^r)=\HH(\overline {b^r})$. Remember that 
\[
\frac{1-|b|^2}{1-|b^r|^2}\asymp 1 
\]
and thus there are two constants $c_1,c_2>0$ such that 
\[
c_1(I-T_{\overline b}T_{b})\leqslant I-T_{\overline{b^r}}T_{b^r}\leqslant c_2(I-T_{\overline b}T_b).
\]
The Douglas factorization theorem implies that $\HH(\bar b)=\HH(\overline{b^r})$ which concludes the proof. 
\end{proof}

\begin{Remark}
In Theorem \ref{Thm1} we see from the above proofs that one can add the condition $\HH(b) = \HH(\overline{b})$ to the list of equivalent conditions. 
\end{Remark}

\section{The contents of $\HH(b)$}

We now can give the proof of Theorem \ref{MainThm}. Indeed, statement (1) of the theorem follows from Lemma \ref{a1bCP} and Theorem \ref{Thm1}. Let us consider statement (2).

In \cite{MR2473631} it was shown, for an outer function $b$, that if $\zeta \in \T$ and 
\begin{equation} \label{FM-AC}
\int_{\T} \frac{|\log|b(w)||}{|w - \zeta|^{2 n + 2}} dm(w) < \infty,
\end{equation}
then every function in $\HH(b)$, as well as its derivatives up to order $n$, has a finite non-tangential limit at $\zeta$.

Recalling the notation $v_{r,\lambda}^\ell$ for the $\ell$-th derivative in the variable $\overline{\lambda}$ of the reproducing kernel in $\HH(q^r )$, the results of \cite{MR2473631} also show that
 $v_{r,\zeta}^\ell\in \HH(q^r)$, $0\leqslant \ell\leqslant n$, and
$$f^{(\ell)}(\zeta) = \langle f, v^{\ell}_{r,\zeta}\rangle_{q^r},\quad f\in \HH(q^r),
\ 0\leqslant \ell\leqslant n.$$ 


Let us check condition \eqref{FM-AC} for our situation.
Since 
$q$ is rational its Pythagorean mate $a$ is also rational and can be written as 
\begin{equation}\label{a-prod-form}
a(z) = s(z) \prod_{j = 1}^{n} (z - \zeta_{j})^{m_j},
\end{equation}
 where $s$ is a rational function whose poles and zeros lie on the complement of $\D^{-}$. 
Pick $w = e^{i t}$ near one of the zeros $\zeta_{j} = e^{i \theta_j}$ of $a$. Then 
\begin{eqnarray*} 
|\log|q^r(e^{i t})|| &\asymp& |\log |q(e^{i t})|^2|
 = |\log (1 - |a(e^{i t})|^2)|\\
&\asymp&|a(e^{i t})|^2\asymp |e^{it} - e^{i \theta_j}|^{2 m_j}\\
\end{eqnarray*}
This means that for $t$ near $\theta_j$ we have 
$$\frac{|\log |q^r(e^{i t})||}{|e^{i t} - e^{i \theta_j}|^{2 (m_j - 1) + 2}} \asymp 1$$ and so, by \eqref{FM-AC},
every function in $\HH(b)$ as well as its derivatives up the order $m_j-1$
admits non-tangential limits at $\zeta_j$, and
$v^{\ell}_{r,\zeta_j} \in \HH(q^r)$ for all $0 \leqslant \ell \leqslant m_{j} - 1$.

The following interesting observation will be very useful in the proof of our
main theorem.

\begin{Lemma}\label{LemMult}
Suppose $a(z)=\prod_{j=1}^n(z-\zeta_j)^{m_j}$, where $\zeta_j\in\T$ and $m_j$ is
the corresponding multiplicity. If the non-tangential limits if an $f=T_{\overline{a}}g\in \MM(\overline{a})$, along with the non-tangential limits of its derivatives up to order $m_j-1$, vanish at every point $\zeta_j$, $j=1,\ldots,n$, then 
\begin{equation}\label{gsarezero}
 \widehat{g}(0)=\widehat{g}(1)=\cdots=\widehat{g}(N-1)=0,
\end{equation}
where $N=\sum_{j=1}^nm_j$.
\end{Lemma}

\begin{proof}[Proof of Lemma]
We prove \eqref{gsarezero} as follows. Consider the kernels 
$$k_{\lambda, \ell}(z) 
= c_{\ell} \frac{z^{\ell}}{(1 - \overline{\lambda} z)^{\ell + 1}},$$
where $c_\ell$ is adjusted so that
these are the reproducing kernels for $\ell$-th derivatives at point $\lambda\in\D$ in the Hardy space $H^2$, that is to say, 
$$f^{(\ell)}(\lambda) = \langle f, k_{\lambda, \ell}\rangle_{H^2} = \int_{\T} f(\zeta) \overline{k_{\lambda, \ell}(\zeta)} dm(\zeta), \qquad f \in H^2.$$ 
Observe, for $1\leqslant j\leqslant n$ and $0 \leqslant \ell \leqslant m_j - 1$, that 
\begin{align*}
a(z) k_{t\zeta_j, \ell}(z) & = c_{\ell} \frac{z^{\ell} (z- \zeta_j)^{m_j}}{(1 - t\overline{\zeta_j} z)^{\ell + 1}}
 \prod_{k\neq j}(z-\zeta_k)^{m_k}\\
& = c_{\ell} z^{\ell} (z - \zeta_j)^{m_j - (\ell + 1)} \left(\frac{z - \zeta_j}{1 - t\overline{\zeta_j} z}\right)^{\ell + 1}
  \prod_{k\neq j}(z-\zeta_k)^{m_k}.
\end{align*}
Writing
\[
 \frac{z - \zeta_j}{1 - t\overline{\zeta_j} z}=
 -\zeta_j\left(1-\overline{\zeta_j}z\frac{1-t}{1-t\overline{\zeta_i}z}\right),
\]
we see that $a(z) k_{t\zeta_j, \ell}(z)$ is uniformly bounded in $z\in\D$ and $t\in [0,1)$,
and moreover
\[
 \frac{z - \zeta_j}{1 - t\overline{\zeta_j} z} \to  -\overline{\zeta_j},\quad
t\to 1,
\]
for every $z$.
Thus, by the dominated convergence theorem, 
$$a k_{t, \ell} \to c z^{\ell} (z - \zeta_j)^{m_j - (\ell + 1)} 
  \prod_{k\neq j}(z-\zeta_k)^{m_k}$$ 
in the norm of $H^2$, where $c$ is some non zero constant depending on $\ell$ and $j$.

Choose any function $f=T_{\overline{a}}g\in\MM(\overline{a})$ 
with
$(T_{\overline{a}}g)^{(\ell)}(\zeta_j)=0$ for all  $1\leqslant j\leqslant n$, $0\leqslant \ell
\leqslant m_j-1$. Recall that $\MM(\overline{a}) \subset \HH(b)$ and so $f$, as well as all its derivatives up to order $m_j - 1$, admits non-tangential limits at $\zeta_j$ for all $1 \leqslant j \leqslant n$.
Then
\begin{align*}
0 & = (T_{\overline{a}} g)^{(\ell)}(\zeta_j)
 = \lim_{t \to 1^{-}} (T_{\overline{a}} g)^{(\ell)}(t\zeta_j)
 = \lim_{t\to 1^{-}} \langle T_{\overline{a}} g, k_{t\zeta_j, \ell}\rangle_{H^2}\\
 & = \lim_{t \to 1^{-}} \langle g, a k_{t\zeta_j, \ell}\rangle_{H^2}\\
& = \overline{c} \langle g, c z^{\ell} (z - \zeta_j)^{m_j - (\ell + 1)} 
  \prod_{k\neq j}(z-\zeta_k)^{m_k}\rangle_{H^2}.
\end{align*}
In order to prove the lemma, it suffices to show that the set 
\[
 \left\{
 \varphi_{j,\ell}(z):=z^{\ell}(z-\zeta_j)^{m_j-(\ell+1)}\prod_{k\neq j}(z-\zeta_k)^{m_k}\right\},
\]
where $j=1,\ldots,n$ and $\ell=0,\ldots,  m_j -1$,
is a basis for the space of polynomials of degree at most $N-1$. Clearly each $\varphi_{j,\ell}$ is a polynomial
of degree $N-1$ and there are $N-1$ functions $\varphi_{j,\ell}$. 
It remains to show that the elements of this family are linearly independent. 
Obviously, for fixed $1\leqslant r\leqslant n$ and $0\leqslant k\leqslant m_r-1$, we have 
\[
 \varphi_{j,\ell}^{(k)}(\zeta_r)=0,\quad j\neq r,\ 0\leqslant k\leqslant m_r-1,
\]
and
\begin{eqnarray}\label{triangle}
  \varphi_{r,\ell}^{(k)}(\zeta_r)=0,\quad  0\leqslant k\leqslant m_r-(\ell+2).
\end{eqnarray}
In particular, if $\sum_{j,\ell}\alpha_{j,\ell}\varphi_{j,\ell}=0$, then,
for fixed $r$ and $0\leqslant k\leqslant m_r-1$, $\sum_{j,\ell}\alpha_{j,\ell}\varphi_{j,\ell}^{(k)}(\zeta_r)=0$
which reduces to $\sum_{\ell}\alpha_{r,\ell}\varphi_{r,\ell}^{(k)}(\zeta_r)=0$. 
Writing $\varphi_{j,\ell}(z)=(z-\zeta_j)^{m_j-(\ell+1)}p_{j,\ell}(z)$, where $p_{j,\ell}$ does not vanish
at $\zeta_j$, Leibniz's formula gives
\begin{eqnarray*}
 \lefteqn{\varphi_{j,\ell}(z)^{m_j-(\ell+1)}(z)=\sum_{k=0}^{m_j-(\ell+1)} \binom{m_j-(\ell+1)}{k}}\\
 &&\qquad \times\frac{(m_j-(\ell+1))!}{(m_j-(\ell+1)-k)!}(z-\zeta_j)^{m_j-(\ell+1)-k}
  p_{j,\ell}^{(m_j-(\ell+1)-k)}(z).
\end{eqnarray*}
Evaluating this expression at $\zeta_j$ makes all terms of the sum vanish except for
$k=m_j-(\ell+1)$, and thus
\[
 \varphi_{j,\ell}(z)^{m_j-(\ell+1)}(\zeta_j)=(m_j-(\ell+1))!p_{j,\ell}(\zeta_j)\neq 0.
\]
This together with
\eqref{triangle} generates a triangular system of linear equations with non-zero diagonal entries.
Thus $\alpha_{r,\ell}=0$, $0\leqslant \ell \leqslant m_r-1$.
\end{proof}

We are now in a position to prove Theorem \ref{MainThm}.

Our arguments so far yield
\begin{equation}\label{ToShow1}
\HH(q^r) = \MM(\overline{a_r})
\end{equation}
and 
\begin{equation}\label{ToShow}
\MM(\overline{a_r})
 =\MM(\overline{a}) \supset \MM(a) +\bigvee \{v^{\ell}_{r,\zeta_j}: 1 \leqslant j \leqslant n, 0 \leqslant \ell \leqslant m_j - 1\}.
\end{equation}

First we show that the sum is orthogonal in the $\HH(q^r)$ inner product:
$$v^{\ell}_{r,\zeta_j} \perp \MM(a), \qquad 1\leqslant j\leqslant n,\,0 \leqslant \ell \leqslant m_j - 1.$$
Indeed, for each $f \in \HH(q^r)$ the radial limits $f^{(\ell)}(t \zeta_j)$ exist as $t \to 1^{-}$. Since 
$$f^{(\ell)}(t \zeta_j) = \langle f, v^{\ell}_{r, t\zeta_j}\rangle_{q^r},$$ we can apply the principle of uniform boundedness to see that $\|v_{r,t \zeta_j}^{\ell}\|_{q^r}$ is uniformly bounded as $t \to 1^{-}$. 
Since $v^{\ell}_{r, t\zeta_j}$ converges pointwise to $v^{\ell}_{r,\zeta_j}$ as $t \to 1^{-}$ we see that $v^{\ell}_{r, t \zeta_j}$ converges weakly to $v^{\ell}_{r,\zeta_j}$. Thus, since 
$v^{\ell}_{r,t\zeta_j}$ reproduces the $\ell$-th derivative of $\HH(q^r)$-functions at point $t\zeta_j$, for any $g\in H^2$, we have
\begin{align*}
\langle a g, v^{\ell}_{r,\zeta_j}\rangle_{q^r} 
& = \lim_{t \to 1^{-}} \langle a g, v^{\ell}_{r, t \zeta_j} \rangle_{q^r}
 = \lim_{t\to 1}(a g)^{(\ell)}(t \zeta_j)\\
& = \lim_{t\to 1} \sum_{p = 0}^{\ell} {\ell \choose p} a^{(p)}(t \zeta_j) g^{(\ell - p)}(t \zeta_j).
\end{align*}
Using the estimate
$$|a^{(p)}(t \zeta_j)| \lesssim (1 - t)^{m_j - p}$$ along with the following standard $H^2$ estimate on the growth of the derivative of an $H^2$ function 
$$|g^{(\ell - p)}(t \zeta_j)| \lesssim \frac{1}{(1 - t)^{(\ell - p) + 1/2}},$$ we see that 
$$|a^{(p)}(t \zeta_j) g^{(\ell - p)}(t \zeta_j)| \lesssim (1 - t)^{m_j - p - ((\ell - p) + 1/2)}.$$
But since 
$0 \leqslant \ell \leqslant m_j - 1$ we see that 
$$m_j - p - ((\ell - p) + 1/2) =m_j-\ell-\tfrac{1}{2}\geqslant \tfrac{1}{2}$$ and so 
$$\lim_{t \to 1^{-}} |a^{(p)}(t\zeta_j) g^{(\ell - p)}(t \zeta_j)| = 0.$$
Thus $\langle a g, v^{\ell}_{r,\zeta_j}\rangle_{q^r} =0$ and $v^{\ell}_{r,\zeta_j} \perp \MM(a)$ in $\HH(q^r)$, for all $0 \leqslant \ell \leqslant m_j - 1$. 

This upgrades \eqref{ToShow1} and \eqref{ToShow} to 
\begin{equation}\label{ToShow2}
\HH(q^r) = \MM(\overline{a}) \supset \MM(a) \oplus \bigvee  \{v^{\ell}_{r,\zeta_j}: 1 \leqslant j \leqslant n, 0 \leqslant \ell \leqslant m_j - 1\},
\end{equation}
and orthogonality is with respect to the norm in $\HH(q^r)$.

To show 
equality in \eqref{ToShow2}, our second step is to show that if
$f  
\in \MM(\overline{a})$ and $f\perp  v^{\ell}_{r, \zeta_j}$ 
for all $1 \leqslant j \leqslant n, 0 \leqslant \ell \leqslant m_j - 1$, then $f \in \MM(a)$. Since $\MM(\overline{a}) = T_{\overline{a}} H^2$ this is equivalent to prove that if $g \in H^2$ and 
$$0 = (T_{\overline{a}} g)^{(\ell)}(\zeta_j) = \lim_{t \to 1^{-}} (T_{\overline{a}} g)^{(\ell)}(t \zeta_j)$$ for all $1 \leqslant j \leqslant n, 0 \leqslant \ell \leqslant m_j - 1$ then $T_{\overline{a}} g \in \MM(a)$. To simplify matters a bit, let us recall the formula for $a$ from \eqref{a-prod-form}.  Since $s$ is a rational function with zeros and poles outside $\D^{-}$ then certainly the Toeplitz operators $T_{1/s}$ and $T_{\overline{1/s}}$ are invertible,
and so $\MM(\overline{a})=\MM(\overline{a/s})$. We can therefore make the simplifying assumption that
$$a(z) = \prod_{j = 1}^{n} (z - \zeta_{j})^{m_j}.$$

We will show that 
\begin{equation}\label{firststep}
(T_{\overline{a}} g)^{(\ell)}(\zeta_j) = 0, \; 1 \leqslant j \leqslant n,  0 \leqslant \ell \leqslant m - 1 \implies T_{\overline{a}} g \in a H^2.
\end{equation}
With $N = \sum_{j = 1}^{n} m_j$, one can verify the identify 
$$\overline{a(\zeta)} = \overline{\zeta}^Na(\zeta)\prod_{j=1}^n(-\overline{\zeta_j})^{m_j}, \qquad \zeta \in \T.$$ Thus
$$T_{\overline{a}} g = \prod_{j=1}^n(-\overline{\zeta_j})^{m_j}P_{+}(a \overline{\zeta}^N g).$$ 
By Lemma \ref{LemMult}, we have
$\widehat{g}(0) = \widehat{g}(1) = \cdots = \widehat{g}(N-1) = 0$,
which shows that $\overline{\zeta}^N g \in H^2$ and so 
$$T_{\overline{a}} g = \Big(\prod_{j=1}^n(-\overline{\zeta_j})^{m_j} \Big)
P_{+}(a \overline{\zeta}^N g) \in a H^2.$$
This completes the proof. 
{{\hfill\qedsymbol}}

\section{Examples}\label{Ex}

\begin{Example}\label{FirstEx}
Consider the function 
$$q(z) = \frac{1}{2} (1 + z)$$ and notice that $q$ is outer and $\|q\|_{\infty}=1$.
One can easily guess the Pythagorean mate for $q$ to be 
$a(z) = \frac{1}{2} (1 - z)$.
The function $a(z)$ has one zero of order $1$ at $z = 1$ and a computation reveals that 
$$v^{0}_{1,1}(z) = \frac{1 - \overline{q(1)} q(z)}{1 - z} = \frac{1}{2}.$$
In this case 
$$\HH(q) = (z - 1) H^2 \oplus \C.$$ Moreover, for any $r > 0$ we get $\HH(q^r) = \HH(q)$ and 
$$\HH(q^r) = (z - 1) H^2 \dotplus \C
 =(z-1)H^2\oplus \C\frac{1-\Big(\dfrac{ 1+z}{2}\Big)^r}{1-z}.$$
\end{Example}

For more general $q$ we need to review the proof of the F\'{e}jer-Riesz theorem which says that if 
$$w(e^{i \theta}) = \sum_{j = -n}^{n} c_j e^{i j \theta}$$ is a non-zero trigonometric polynomial which assumes non-negative values for all $\theta$, then there is an analytic polynomial 
$$p(z) = \sum_{j = 0}^{n} a_j z^j$$ 
so that $w(e^{i \theta}) = |p(e^{i \theta})|^2$.
Since the proof gives us the algorithm for computing $p$, we give a quick sketch. Indeed, as a function of the complex variable $z$, we see that if 
$$w(z) = \sum_{j = -n}^{n} c_j z^{j}$$ 
then 
$\overline{w(1/\overline{z})} = w(z)$, $z\in\T$. 
Assuming that $c_{-n} \not = 0$ we see that $s(z) = z^n w(z)$, $z\in\C$, is a polynomial of degree $2 n$ and the roots of $s$ occur in the pairs $\alpha, 1/\overline{\alpha}$ of equal multiplicity. It follows that 
$$w(z) = c \prod_{j = 1}^{n} (z - \alpha_j) (\frac{1}{z} - \overline{\alpha_j})$$ for some positive constant $c$ and where $\alpha_1, \ldots, \alpha_n$ satisfy $|\alpha_j| \geqslant 1$ for $1 \leqslant j \leqslant n$. The desired polynomial $p$ is 
$$p(z) = \sqrt{c} \prod_{j = 1}^{n} (z - \alpha_j).$$ 
Note that $p$ is zero free in $\D$ and we can multiply $p$ by a unimodular constant so that $p(0)>0$. 

Recall from the proof of Lemma \eqref{mate} that if $q = p_1/p_2$ is rational then the Pythagorean mate $a$ for $q$ is given by $a = p/p_2$,
where $p$ is the analytic polynomial (guaranteed by the F\'{e}jer-Riesz theorem) which satisfies 
$|p(e^{i \theta})|^2 = w(e^{i \theta}) = |p_2(e^{i \theta})|^2 - |p_1(e^{i \theta})|^2 \geqslant 0$, and $p$ is chosen to that $a(0) >0$.
\begin{Example}
Consider the function 
$$q(z) = \frac{1}{2} (1-z) (1 + z)$$ and note that $q \in \mathbf{b}(H^{\infty})$ and is outer. A computation shows that 
$$1 - |q(e^{i t})|^2 = \frac{1}{4} e^{-2 i t}+\frac{1}{4} e^{2 i t}+\frac{1}{2}.$$
Define 
$$w(z) = \frac{z^{-2}}{4} + \frac{z^2}{4} + \frac{1}{2}$$ 
and 
$$s(z) = z^2 w(z) = \frac{z^4}{4}+\frac{z^2}{2}+\frac{1}{4} = \frac{1}{4} (z - i)^2 (z + i)^2.$$
Notice how the zeros occur in pairs $i= 1/\overline{i}$ and $-i= 1/\overline{-i}$ 
as guaranteed by the above proof of the F\'{e}jer-Riesz theorem. 
Thus the Pythagorean mate $a$ for $q$ is of the form 
$a(z) = c (z - i) (z + i)$
for some $c$ adjusted so that $a(0) > 0$ and $1 - |q(e^{i \theta})|^2 = |a(e^{i \theta})|^2$. One can check by direct calculation that $c = 1/2$ works and so 
$a(z) = \tfrac{1}{2} (z - i) (z + i)$.
Of course the exact value of $c$ is not important for our calculations since we only need to identify the zeros of $a$ along with their multiplicities. 

The zeros of $a$ are at $z = i$ and $z = -i$ and each has order one. Thus 
$$\HH(q) = (z - i) (z + i) H^2 \oplus \bigvee \{v^{0}_{1,i},v^{0}_{1,-i}\},$$
where the kernels can be computed directly as 
$$v^{0}_{1,i}(z) = \frac{1}{2 i} (z + i), \qquad v^{0}_{1,-i}(z) = \frac{1}{2 i} (z - i).$$
Again, as in the previous example, $\HH(q^r)  = \HH(q)$ and so 
$$\HH(q^r) = (z - i) (z + i) H^2 \dotplus \bigvee \{z + i, z - i\}.$$
\end{Example}

\begin{Example}
Consider the  function 
$$q(z) = \frac{1}{4} (z+1)^2$$ and note that $q$ is outer and belongs to $\mathbf{b}(H^{\infty})$. Following our Fejer-Riesz computations as in the previous example, note that 
$$1 - |q(e^{i t})|^2 = -\frac{e^{-i t}}{4}-\frac{e^{i t}}{4}-\frac{1}{16} e^{-2 i t}-\frac{1}{16} e^{2
   i t}+\frac{5}{8}.$$
 Define 
 $$w(z) = -\frac{z^2}{16}-\frac{1}{16 z^2}-\frac{z}{4}-\frac{1}{4 z}+\frac{5}{8}$$ and 
 \begin{align*}
 s(z) & = z^2 w(z)\\
 & = -\frac{z^4}{16}-\frac{z^3}{4}+\frac{5 z^2}{8}-\frac{z}{4}-\frac{1}{16}\\
 & = -\frac{1}{16} (-1 + z)^2 (1 + 6 z + z^2).
 \end{align*}
 The zeros of $s$ are at 
 $$z = -1, z = -1, z = -3 - 2 \sqrt{2} \approx -5.82843, z = -3 + 2 \sqrt{2} \approx -0.171573.$$ Notice how these roots occur in the pairs $\alpha, 1/\overline{\alpha}$. The function $a$ is then 
 $a(z) = c (z - 1) (z + 3 + 2 \sqrt{2})$ for some appropriate constant $c$. There is one zero of $a$ at $z = 1$ with multiplicity one and so 
 $$\HH(q) = (z - 1) H^2 \oplus \C v^{0}_{1,1}(z).$$ The kernel can be computed to be 
 $$v^{0}_{1,1}(z) = \frac{z+3}{4}.$$ As in our previous examples, note that 
 $$\HH(q^r) = (z - 1) H^2 \dotplus \C (z + 3).$$
 
Observe that the $q$ from this example is the square of the $q$ from Example \ref{FirstEx} and thus the corresponding spaces should be the same. Indeed, a little algebra will show that 
$$(z - 1) H^2 \dotplus \C = (z - 1) H^2 \dotplus \C (z + 3).$$
\end{Example} 

\begin{Example}
Reversing the roles of $a$ and $q$ in the preceding example:
\[
 a(z)=\frac{1}{4}(z+1)^2,\qquad q(z)=c(z-1)(z+3+2\sqrt{2}), 
\]
with suitable $c$ 
so that $\|q\|_{\infty}=1$ (the maximum modulus on $\D^-$ being attained at $-1$, one
has $c=(4(1+\sqrt{2}))^{-1}$, and $q(-1)=-1$ corresponding to the normalization
$q(0)>0$), we obtain a function $a$ with double zero,
and so
\[
 \HH(q)=(z+1)^2H^2\oplus \bigvee\{v_{1,-1}^0,v_{1,-1}^1 \}
\]
where
\[
 v_{1,-1}^0(z)=\frac{1-\overline{q(-1)}q(z)}{1-\overline{(-1)}z}=\frac{1+q(z)}{1+z}.
\]
Using the facts that $q(-1)=-1$, $q'(z)=c(2+2\sqrt{2})$, and $q'(-1)=-1/2$, we obtain
\[
 v_{1,-1}^1(z)=\frac{\frac{1}{2}q(z)(1+z)+z(1+q(z))}{(1+z)^2}
 =\frac{1}{2}\frac{q(z)(1+3z)+2z}{(1+z)^2}.
\]
\end{Example}

\begin{Question}
So far we have computed the exact contents of $\HH(b)$ when $b$ outer, rational, and non-extreme. Can one compute the contents of $\HH(b)$ when $b$ is outer and extreme. For example if $b$ is the outer function corresponding to the outer function which satisfies $|b(e^{i \theta})| = 1$ for $0 \leqslant \theta \leqslant \pi$ and $|b(e^{i \theta})| = \tfrac{1}{2}$ for $\pi < \theta < 2 \pi$, can one describe the functions in $\HH(b)$?
\end{Question}


\bibliographystyle{plain}

\bibliography{references}

\end{document}